\documentclass[]{amsart}
\usepackage{amssymb}
\usepackage{amsmath}
\usepackage{indentfirst}
\usepackage{graphicx}
 \usepackage{mathrsfs}

\newtheorem{theorem}{Theorem}[section]
\newtheorem{lem}[theorem]{Lemma}
\newtheorem{prop}[theorem]{Proposition}

\theoremstyle{definition}
\newtheorem{defi}[theorem]{Definition}
\newtheorem{remark}[theorem]{Remark}

\newcommand{\R}{\ensuremath{\mathbb{R}}}

\newcommand{\Z}{\ensuremath{\mathbb{Z}}}

\newcommand{\Col}{\mathrm{Col}}
\newcommand{\mK}{\mathcal{K}}
\newcommand{\mT}{\mathcal{T}}
\newcommand{\mD}{\mathcal{D}}
\newcommand{\mS}{\mathcal{S}}
\newcommand{\mP}{\mathcal{P}}
\newcommand{\sD}{\mathscr{D}}

\begin{document}


\renewcommand{\subjclassname}{%
    \textup{2020} Mathematics Subject Classification}

\subjclass[2020]{Primary 57K45; Secondary 57K12.}
\date{\today}
\keywords{Bridge number, bridge trisection, surface link, kei, coloring}


\title[The bridge number of surface links and kei colorings]
{The bridge number of surface links and kei colorings}


\author{Kouki Sato}
\address{Graduate School of Mathematical Sciences, 
University of Tokyo, 3-8-1 Komaba Meguro, 
Tokyo 153-8914, Japan}
\email{sato.kouki@mail.u-tokyo.ac.jp}

\author{Kokoro Tanaka}
\address{Department of Mathematics, Tokyo Gakugei University, 
Nukuikita 4-1-1, 
Koganei, Tokyo 184-8501, Japan}
\email{kotanaka@u-gakugei.ac.jp}




\begin{abstract}
Meier and Zupan introduced bridge trisections of surface links in $S^4$
as a 4-dimensional analogue to bridge decompositions of classical links, which gives a
numerical invariant of surface links called {\it the bridge number}.
We prove that there exist infinitely many surface knots with bridge number $n$ 
for any integer $n \geq 4$.
To prove it, we use colorings of surface links by keis
and give lower bounds for the bridge number of surface links.
\end{abstract}

\maketitle



\section{Introduction}

A {\it surface link} is a smoothly embedded closed 
surface 
in $S^4$, which may be disconnected or non-orientable.
A surface link is called a {\it surface knot} if it is connected.
In \cite{MZ17}, Meier and Zupan introduced {\it bridge trisections} of surface links as a 4-dimensional analogue to bridge decompositions of classical links. 
Moreover, any bridge trisection is associated to a {\it tri-plane diagram},
which consists of three diagrams for trivial 1-dimensional tangles. Using these notions, Meier and Zupan defined the {\it bridge number} as a new geometrical complexity of surface links.
Indeed, it is proved in \cite{MZ17} that if the bridge number of a surface link $\mK$ is equal or less than $3$, then $\mK$ is one of trivially embedded $S^2$, $\R P^2$ and $T^2$.

One of fundamental problems about the bridge number is whether there exists a surface link  with bridge number $n$ for any positive integer $n$. 
For the cases where $n \equiv 1 \mod 3$, this problem is resolved affirmatively in \cite{MZ17}, while it remained open in general. The aim of this paper is to resolve this  problem affirmatively for general cases.
\begin{theorem}
\label{thm: main}
For any integer $n \geq 4$, there exist infinitely many distinct surface knots with bridge number $n$.
\end{theorem}
Here we remark that the bridge number of a surface link $\mK$ is congruent to 
$-\chi(\mK)$ modulo 3, where $\chi(\mK)$ is the Euler characteristic of $\mK$.
Any of our examples for proving Theorem~\ref{thm: main} is homeomorphic to one of $S^2$, $\R P^2$ and $T^2$.
To prove Theorem~\ref{thm: main}, we introduce {\it kei colorings} of tri-plane diagrams,
and show that such colorings coincide with the original kei colorings of surface links.
As its application, if we consider a finite kei, we have the following lower bound for the bridge number.

\begin{theorem}
\label{thm: lower bound}
For any surface link $\mK$ and finite kei $X$, we have the inequality
\[
b(\mK) \geq 3 \log_{\# X} (\# \Col_X(\mK)) -\chi(\mK),
\]
where $b(\mK)$ is the bridge number of $\mK$ and $\# \Col_X(\mK)$
the number of $X$-colorings of $\mK$.
\end{theorem}
For the $m$-twist-spinning $\mS_{m}(K)$ of a classical knot $K$, 
the relation between $\# \Col_X(\mS_{m}(K))$ and $\# \Col_X(K)$ is studied 
by Asami and Satoh \cite{AS05}. Their study enables us to give a sufficient condition for the equality in Theorem~\ref{thm: lower bound} for $m$-twist spun knots and their stabilizations.
Here, we denote by $P$ either one of the two trivially embedded $\R P^2$'s in $S^4$ (with normal Euler number $\pm 2$) and by $T$ a trivially embedded torus in $S^4$.
For any integer $p \geq 2$, we denote by $R_p$ the dihedral kei of order $p$.
(Note that the coloring by $R_p$ is coincident with the Fox {\it $p$-coloring} 
\cite{Fox, Fox2} for classical links.)
\begin{theorem}
\label{thm: twist spun}
For a classical knot $K$, suppose that there exists an odd integer $p > 1$ with $b(K) = \log_{\#R_p}(\# \Col_{R_p}(K))$, where $b(K)$ denotes the bridge number of $K$.
Then for any $m \in \Z$, we have the following equalities:
\begin{itemize}
\item 
$b(\mS_{2m}(K) ) = 3 b(K) - 2$,
\item
$b(\mS_{2m}(K) \# P) = 3 b(K) -1$, and
\item
$b(\mS_{2m}(K) \# T) = 3 b(K)$.
\end{itemize}
\end{theorem}
Theorem~\ref{thm: main} immediately follows from Theorem~\ref{thm: twist spun}. 
\begin{proof}[Proof of Theorem~\ref{thm: main}]
Since the $k$ times connected sum $\#_{k}T_{2,q}$ of the $(2,q)$-torus knot satisfies 
$$b(\#_k T_{2,q}) = \log_q(\# \Col_{R_q}(\#_kT_{2,q}))= k+1, $$
for any odd $q > 1$, 
Theorem~\ref{thm: twist spun} can be applied for $\#_{k}T_{2,q}$. 
Moreover, if an odd number $q'$ is less than $q$, 
then it follows from a direct calculation that $\# \Col_{R_q}(\#_kT_{2,q'}) < q^{k+1}$. 
Hence their $0$-twist spuns $\{\mS_{0}(\#_k T_{2,q})\}_{q > 1}$ 
are mutually distinct, since 
their coloring numbers are the same as those of the original classical knots, respectively. 
%
\end{proof}
We remark that Theorem~\ref{thm: twist spun} also holds even if we replace $R_p$
with any finite {\it faithful} kei. For more details, see Section~\ref{sec: twist spun}. 
Here we also mention Meier-Zupan's question \cite[Question~5.2]{MZ17}, which asks
whether the equality $b(\mS_{m}(K) ) = 3 b(K) - 2$ holds for any $K$ and $m \neq \pm 1$.
Theorem~\ref{thm: twist spun} gives a sufficient condition for the equality,
while it can be applied to only the cases where $m$ is even.

Finally, we remark that if a given surface link is oriented, then we can define {\it quandle colorings} of its oriented tri-plane diagrams and can prove the coincidence of the colorings with the original quandle colorings. Such quandle colorings would be useful to approach
Meier-Zupan's question for odd $m$.

\section{Kei colorings of tri-plane diagrams}
In this section, we introduce {\it kei colorings} of tri-plane diagrams, and show that
such colorings coincide with  original kei colorings of surface links. 
As a consequence of those arguments, we prove Theorem~\ref{thm: lower bound}.

\subsection{A review of tri-plane diagrams}
We first recall {\it bridge trisections},  {\it the bridge number} and {\it tri-plane diagrams} of surface links. 

A {\it $0$-trisection} of $S^4$ is a decomposition 
$S^4 = X_1 \cup X_2 \cup X_3$,
such that
\begin{enumerate}
\item $X_i$ is a 4-ball,
\item $B_{ij} = X_i \cap X_j = \partial X_i \cap \partial X_j$ is a 3-ball, and
\item $\Sigma = X_1 \cap X_2 \cap X_3 = B_{12} \cap B_{23} \cap B_{31}$ is a 2-sphere.
\end{enumerate}
A {\it trivial $c$-disk system} is a pair $(X,\mD)$ where
$X$ is a 4-ball and $\mD$ is a collection of $c$ properly embedded disks 
in $X$ which are
simultaneously isotopic into the boundary of the 4-ball $X$.
Using these terminologies, a {\it bridge trisection} of a surface link is defined as follows.
\begin{defi}[\text{\cite[Definition~1,2]{MZ17}}]
A {\it $(b; c_1, c_2, c_3)$-bridge trisection} $\mT$ of a surface link 
$\mK \subset S^4$ is
a decomposition of the form 
$(S^4,\mK) = (X_1,\mD_1) \cup (X_2,\mD_2) \cup (X_3,\mD_3)$ such that
\begin{enumerate}
\item $S^4 = X_1 \cup X_2 \cup X_3$ is a $0$-trisection of $S^4$,
\item $(X_i, \mD_i)$ is a trivial $c_i$-disk system, and
\item $(B_{ij}, \alpha_{ij}) = (X_i,\mD_i) \cap (X_j ,\mD_j)$ is a $b$-strand trivial tangle.
\end{enumerate}
When appropriate, we simply refer to $\mT$ as a $b$-bridge trisection.
\end{defi}
It is proved in \cite{MZ17} that every surface link admits a bridge trisection.
The {\it bridge number} $b(\mK)$ of a surface link $\mK$ is defined by 
\[
b(\mK) = \min \{b \mid \text{$\mK$ admits a $b$-bridge trisection}\}.
\]

Next we recall {\it tri-plane diagrams}.
For a $(b;c_1,c_2,c_3)$-bridge trisection $\mT$ of a surface link $\mK$, 
choose the orientation of $B_{ij}$ as a submanifold of $\partial X_i$. 
For each pair of indices, let $E_{ij} \subset B_{ij}$ be an embedded disk with the property that
$e = \partial E_{12} = \partial E_{23} = \partial E_{31}$. We call the union 
$E_{12} \cup E_{23} \cup E_{31}$ a {\it tri-plane}.
We also suppose that the points $\mathbf{p} = \mK \cap \Sigma$ lie in the curve 
$e = E_{12} \cap E_{23} \cap E_{31}$. Choose an orientation of
each $E_{ij}$ so that the $E_{ij}$ induce the same orientation of their common boundary $e$.
Then, a {\it tri-plane diagram} $\mP$ of $\mT$ is a triple of planar tangle diagrams $\mP=(\mP_{12}, \mP_{23}, \mP_{31})$ such that 
$\mP_{ij}$ is obtained as a diagram of $\alpha_{ij}$ in $E_{ij}$.
By definition, there is a canonical identification among $\partial \mP_{ij}$,
and if $\overline{\mP_{ki}}$ denote the mirror image of $\mP_{ki}$, then
$\mP_{ij} \cup \overline{\mP_{ki}}$ is a classical link diagram of a $c_i$-component trivial 
classical link, where the diagram lies in $S^2 = E_{ij} \cup (-E_{ki})$. 
Refer to \cite{MZ17} for more details.

\subsection{Kei colorings of tri-plane diagrams}
Now we introduce {\it kei colorings} of tri-plane diagrams. 

A {\it kei} \cite{Tak43} 
is a non-empty set $X$ equipped with a binary 
operation $(a,b) \mapsto a*b$ such that 
\begin{itemize}
\item[(i)] $a*a=a$ for any $a\in X$, 
\item[(ii)] $(a*b)*b=a$ for each $a, b \in X$, and
\item[(iii)] $(a*b)*c = (a*c)*(b*c)$ for any $a,b,c \in X$.
\end{itemize}
For examples, the {\it dihedral kei of order $p$}, denoted by $R_p$, 
is a kei consisting of the set $\{0,1,\ldots,p-1\}$ with 
the binary operation defined by
$i*j \equiv 2j-i \pmod p$.

Here we recall kei colorings of tangle diagrams.
For a given kei $X$, an {\it $X$-coloring} of a tangle diagram $T$ 
is an assignment of an element of $X$ to each arc in $T$ 
such that $a*b=c$ holds at each crossing, where $a$ and $c$ (resp.\ $b$)
are the colors of under-arcs (resp.\ the over-arc).
Note that $a*b=c$ is equivalent to $c*b=a$ by the axiom (ii) of a kei. 
Any tangle diagram admits a \textit{trivial $X$-coloring} 
where all the arcs are assigned the same element.
We denote by $\Col_{X}(T)$ the set of all $X$-colorings of $T$. 
Now, {\it kei colorings} of tri-plane diagrams are defined as follows.
Here recall that there is a canonical identification among $\partial \mP_{ij}$.
\begin{defi}
\label{def: kei col}
For a given kei $X$, an {\it $X$-coloring} of a tri-plane diagram 
$\mP=(\mP_{12}, \mP_{23}, \mP_{31})$ is a triple of $X$-colorings $C_{ij}$ of $\mP_{ij}$
such that for each end point $p \in \partial \mP_{ij}$, the arcs in $\mP_{ij}$ containing $p$
are colored by the same element.
\end{defi}
We denote by $\Col_{X}(\mP)$ the set of all $X$-colorings of $\mP$.

\begin{remark}
If a given surface link $\mK$ is oriented, then we can define an {\it oriented tri-plane diagram} of $\mK$ so that the orientation of $\mP_{ij}$ is induced from $\partial \mD_i$. Then, we can define {\it quandle colorings} of oriented tri-plane diagrams 
in a way similar to Definition~\ref{def: kei col}.
\end{remark}

\subsection{Relationship to original kei colorings}
Here we show that kei colorings of tri-plane diagrams can be identified with kei colorings of {\it broken surface diagrams}. 
As its application, we prove Theorem~\ref{thm: lower bound}.

For a given surface link $\mK \subset S^4$, take a point 
$p_{\infty} \in S^4 \setminus \mK$ and identify $\R^4$ with $S^4 \setminus \{p_{\infty}\}$.
A {\it broken surface diagram} of a surface link is a generic projection image 
in $\R^3$ where one of the two sheets near the double point 
curve is broken depending on the relative height.
This convention is similar to classical knot diagrams.
A broken surface diagram consists of {\it broken sheets}, that are mutually disjoint 
compact surfaces in $\R^3$. Any surface link admits a broken surface diagram, and any two broken surface diagrams of a surface link are related by a finite sequence of {\it Roseman moves}.
Refer to \cite{CS-book} for more details.

For a given kei $X$, an {\it $X$-coloring} of a broken surface diagram $\mathscr{D}$
is an assignment of an element of $X$ to each broken sheet in $\mathscr{D}$
such that $a*b=c$ holds along 
each double point curve, where $a$ and $c$ (resp.\ $b$)
are the colors of under-sheets (resp.\ the over-sheet). 
Any diagram admits a \textit{trivial $X$-coloring} 
where all the broken sheets are assigned the same element.
We denote by $\Col_X(\mathscr{D})$ the set of $X$-colorings of $\mathscr{D}$.
We remark that each Roseman move from $\mathscr{D}_1$ to $\mathscr{D}_2$ induces 
a bijection from $\Col_X(\mathscr{D}_1)$ to $\Col_X(\mathscr{D}_2)$; 
see, for example, \cite{CKS04, Ros98}.
In particular, for any finite kei $X$, 
the number of the $X$-colorings is an invariant of a surface link $\mK$. Hence we denote the value by $\# \Col_{X}(\mK)$.
It holds that $\# \Col_{X}(\mK) = \# X$ for a trivial surface knot $\mathcal{K}$, 
since its trivial diagram admits only trivial $X$-colorings. 
%

\begin{prop}
\label{prop: diagram}
For any tri-plane diagram $\mP$ of a surface link $\mK$,
there exists a broken surface diagram $\sD(\mP)$ 
of $\mK$ in $\R^4$ such that for any kei $X$,  we have a bijection
\[
\varphi \colon \Col_X(\mP) \to \Col_X(\sD(\mP)).
\]
In particular, if $X$ is finite, then  we have
\[
\# \Col_X(\mP) = \# \Col_X(\mK).
\] 
\end{prop}
\begin{proof}
Suppose that $\mP$ is a tri-plane diagram of a $(b;c_1,c_2,c_3)$-bridge trisection $\mT$.
Take a point $p_{\infty}$ in $e = E_{12} \cap E_{23} \cap E_{31}$ with $p_{\infty} \not\in \mK$.
By taking an orientation preserving diffeomorphism between $S^4 \setminus \{p_{\infty}\}$
and $\R^4$ suitably, we may assume that 
\[
X_j \setminus \{p_{\infty}\}  = 
\left\{(r \cos \theta, r \sin \theta, z, w) \  \middle|\  r \geq 0, 
\frac{2\pi j}{3} \leq \theta \leq 
\frac{2\pi(j+1)}{3} \right\}
\]
for $j = 1,2,3$, and
\[
E_{ij} \setminus \{p_{\infty}\}
= \R^{2+}_k
:= \left\{ \left(r \cos \frac{2\pi j}{3}, r \sin \frac{2\pi j}{3}, z, 0 \right)  \ \middle|\  r \geq 0 \right\}
\]
for any cyclic permutation, denoted by $(i,j,k)$, of $(1,2,3)$ without loss of generality.
Then it also follows that
\begin{itemize}
\item $B_{ij} \setminus \{p_{\infty}\}
= \{ (r \cos (2\pi j /3), r \sin (2\pi j/3), z, w) \mid r \geq 0\}$,
\item $\Sigma \setminus \{p_{\infty}\}$ is the $zw$-plane, and
\item $e \setminus \{p_{\infty}\}$ is the $z$-axis.
\end{itemize}

In particular, $D_i := \mP_{ij} \cup \overline{\mP_{ki}}$ can be regarded as a classical link diagram of the trivial $c_i$-component classical link $\partial \mD_i$, 
where the diagram lies in the plane $\R^{2+}_k \cup (-\R^{2+}_j)$ 
and the orientation of $\R^{2+}_k$ is induced from $E_{ij}$. 
To obtain a broken surface diagram of $\mK$,
it suffices to construct a broken surface diagram $\sD_i$ of the trivial $c_i$-disk system 
$\mD_i = \mK \cap X_i$ in the half space 
\begin{eqnarray*}
(X_i \setminus \{p_{\infty}\}) \cap \R^3 &=&
\left\{(r \cos \theta, r \sin \theta, z, 0) \  \middle|\  r \geq 0, 
\frac{2\pi j}{3} \leq \theta \leq 
\frac{2\pi(j+1)}{3} \right\}\\
& \cong& \R^2 \times [0, \infty)
\end{eqnarray*}
such that $\sD_i \cap (\R^{2+}_k \cup (-\R^{2+}_j)) = D_i$.
Then we will conclude that $\sD(\mP) := \sD_1 \cup \sD_2 \cup \sD_3$ is a
broken surface diagram of $\mK = \mD_1 \cup \mD_2 \cup \mD_3$ such that
$\sD(\mP) \cap \R^{2+}_k = \mP_{ij}$.

To obtain such $\sD_i$, we first take a sequence of Reidemeister moves 
and ambient isotopies of the plane
\[
D_i = D_{i,0} \overset{m_1}{\longrightarrow} D_{i,1} \overset{m_2}{\longrightarrow} \cdots 
\overset{m_n}{\longrightarrow} D_{i, n} = D'_i
\]
such that all components of $D'_i$ simultaneously bound disks in the plane.
Then, after fixing the identification $(X_i \setminus \{p_{\infty}\})\cap \R^3 \cong \R^2 \times [0, \infty)$, we associate to each move $m_k$ a broken surface diagram $\sD_{i,k}$
in $\R^2 \times [k-1, k]$ as follows:
If $m_k$ is an ambient isotopy $F_k \colon \R^2 \times [0,1] \to \R^2$, then we define 
$\sD_{i,k}$ as the image of the map
\[
D_{i,k-1} \times [0,1] \to \R^2 \times [k-1,k], \ (x,t) \mapsto (F_k(x,t), t + k-1).
\]
If $m_k$ is a Reidemeister move, then we take a small disk $\delta$ in $\R^2$ so that $D_{i,k-1} \setminus \delta$ coincides with $D_{i,k} \setminus \delta$, and define 
$\sD_{i,k}$ as the union of $(D_{i,k} \setminus \delta) \times [k-1,k]$ with 
one of the surfaces shown in Figure~\ref{Fig: R moves}.
\begin{figure}[htbp]
\begin{center}
\includegraphics[scale =1]{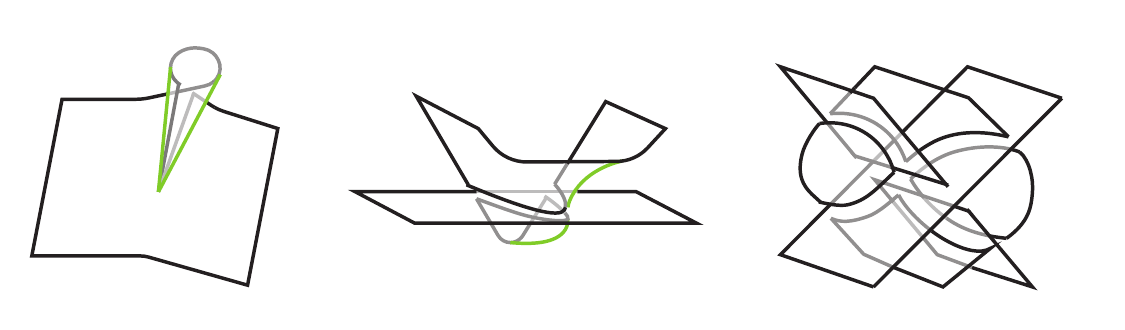}
\caption{Surfaces associated to Reidemeister moves. \label{Fig: R moves}}
\end{center}
\end{figure}
For both cases,
we have $\sD_{i,k} \cap (\R^2\times \{k-1\}) = D_{i,k-1} \times \{k-1\}$
and $\sD_{i,k} \cap (\R^2\times \{k\}) = D_{i,k} \times \{k\}$.

Finally, we define $\sD_{i,n+1}$ as mutually distinct disks smoothly embedded in $\R^2 \times [n,n+1]$ with boundary $D_{i,n} \times \{n\} =D'_i \times \{n\}$.
Now we have a broken surface diagram
\[
\sD_i := \sD_{i,0} \cup \cdots \cup \sD_{i,n+1}
\]
with boundary $D_i$. From the construction of $\sD_i$, it is obvious that an embedded disks $\mD'_i$ in $X_i \setminus \{p_{\infty}\}$ whose projection image gives $\sD_i$ are simultaneously isotopic into 
$\partial X_i \setminus \{p_{\infty}\}$, and hence $\mD'_i$ is a trivial $c_i$-disk system with boundary
$\partial \mD_i$. As shown in \cite{Kam-book}, 
any two trivial disk systems with the same boundary are isotopic, 
and hence $\mD'_i$
is isotopic to $\mD_i$. In particular, $\sD_i$ is regarded as a broken surface diagram of $\mD_i$.
Now, we have a broken surface diagram $\sD(\mP) = \sD_1 \cup \sD_2 \cup \sD_3$
of $\mK = \mD_1 \cup \mD_2 \cup \mD_3$. 

Next, we give bijections between $\Col_X(\mP)$ and $\Col_X(\sD(\mP))$.
We first define a map
\[
\varphi \colon \Col_X(\mP) \to \Col_X(\sD(\mP))\]
as follows. Let $C=(C_{12},C_{23}, C_{31}) \in \Col_X(\mP)$.
Note that $C_{ki}$ naturally induces a coloring $\overline{C_{ki}}$ of $\overline{\mP_{ki}}$,
and $C_{ij} \cup \overline{C_{ki}}$ defines an $X$-coloring  of the diagram $D_i=D_{i,0}$.
Moreover, each move $m_k$ and $X$-coloring of $D_{i,k-1}$ defines an $X$-coloring of
$D_{i,k}$. Hence we have $X$-colorings of $D_{i,k} \times \{k\}$. 
Here we also note that any broken sheet of $\sD_i$ has non-empty intersection 
to $\R^2 \times \{k\}$ for some $0 \leq k \leq n$.
Now, we define $\varphi(C)$ so that the coloring of each broken sheet is the same as that of 
its intersection to $\bigsqcup_{0 \leq k \leq n} \R^2 \times \{k\}$. 
As shown in Figure~\ref{Fig: R moves2}, we can see that this assignment is well-defined as an $X$-coloring of $\sD(\mP)$.
\begin{figure}[htbp]
\begin{center}
\includegraphics[scale =0.85]{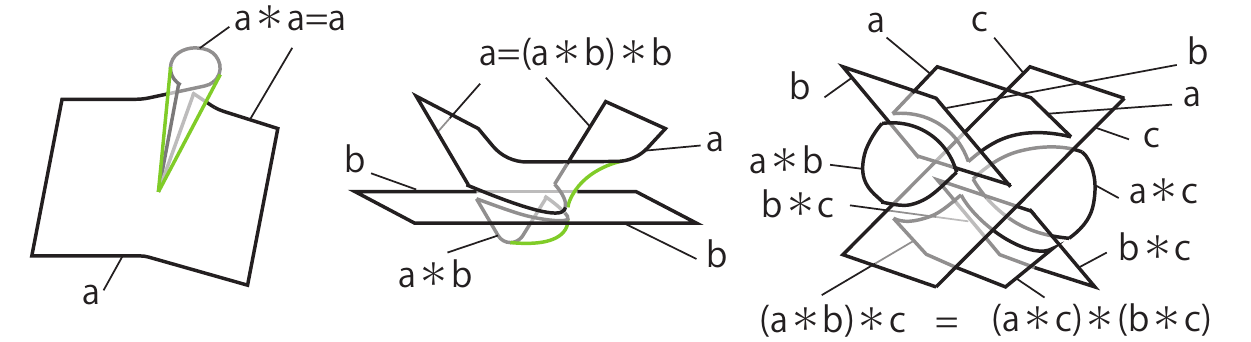}
\caption{The coloring of each broken sheet is uniquely induced from its boundary. \label{Fig: R moves2}}
\end{center}
\end{figure}

We next define a map
\[
\psi \colon \Col_X(\sD(\mP)) \to \Col_X(\mP)
\]
as follows. Let $C \in \Col_X(\sD(\mP))$. Regard $\mP_{ij}$ as embedded in $\sD_i$, and then each arc of $\mP_{ij}$ is lying in a broken sheet of $\sD_i$.
We define $\psi(C)_{ij}$ so that the coloring of each arc is the same as that of the broken sheet containing the arc. Then, we can see that for each end point $p \in \partial \mP_{ij}$,
the arcs in the $\mP_{ij}$ containing $p$ are lying in the same broken sheet of $\sD$,
and hence colored by the same element. Moreover, for each crossing $c$ in $\mP_{ij}$, the arcs near $c$ are lying in the broken sheets near a double point curve, and under-over information is preserved. These imply that $\psi(C) := (\psi(C)_{12}, \psi(C)_{23},\psi(C)_{31})$ is 
a well-defined $X$-coloring of $\mP$.
Now, it is easy to see that both $\psi \circ \varphi$ and $\varphi \circ \psi$ are
the identities, and hence these are bijections.
\end{proof}

By the same arguments as the proof of Proposition~\ref{prop: diagram},
we can also prove the following proposition.
\begin{prop}
For any oriented tri-plane diagram $\mP$ of an oriented surface link $\mK$,
there exists an oriented broken surface diagram $\sD(\mP)$ 
of $\mK$ in $\R^4$ such that for any quandle $Q$,  we have a bijection
\[
\varphi \colon \Col_Q(\mP) \to \Col_Q(\sD(\mP)).
\]
In particular, if $Q$ is finite, then  we have
\[
\# \Col_Q(\mP) = \# \Col_Q(\mK).
\] 
\end{prop}

\subsection{Proof of Theorem~\ref{thm: lower bound}}

To prove Theorem~\ref{thm: lower bound}, we use the following lemma.
\begin{lem}
\label{lem: lower bound}
If a surface link $\mK$ admits a $(b; c_1,c_2,c_3)$-trisection $\mT$,
then for any finite kei $X$, we have
\[
\min\{c_1,c_2,c_3\} \geq \log_{\# X}(\#\Col_X(\mK)).
\]
\end{lem}
\begin{proof}
We first note that for the $c$-component trivial classical link $L_c$ and the finite kei $X$,
we have
\[
\# \Col_{X}(L_c) = (\# X)^c.
\]
Therefore, it suffices to prove that the inequality
\[
\# \Col_X(L_{c_i}) \geq \# \Col_X(\mK)
\]
holds for any $i \in \{1,2,3\}$. Let $\mP =(\mP_{12}, \mP_{23}, \mP_{31})$ be a tri-plane diagram of $\mT$.
Then, we can define a map $\iota \colon \Col_X(\mP) \to \Col_X(L_{c_i})$
by using the colorings $\{C_{ij}, C_{ki}\}$ of the tangle diagrams $\{\mP_{ij}, \mP_{ki}\}$ to 
color the classical link diagram $\mP_{ij} \cup \overline{\mP_{ki}}$.
Here we note that any coloring of a tangle diagram of a trivial tangle
is characterized by the colorings of the end points. In particular, 
for any $C=(C_{12}, C_{23}, C_{31}) \in \Col_X(\mP)$, 
the coloring $C_{jk}$ of $\mP_{jk}$ is uniquely determined by $\{C_{ij}, C_{ki}\}$.
This implies that $\iota$ is injective, and hence 
the desired inequality holds.
\end{proof}

\def\proofname{Proof of Theorem~\ref{thm: lower bound}}
\begin{proof}
Suppose that $\mK$ admits a $(b; c_1, c_2, c_3)$-trisection.
Then $\mK$ has a cell decomposition consisting of $2b$ vertices, $3b$ edges and $(c_1 +c_2 + c_3)$ faces. In particular, we have
\[
\chi(\mK) = 2b -3b + (c_1 + c_2 + c_3) \geq  3 \min \{c_1, c_2, c_3\} -b. 
\]
Now, the desired inequality immediately follows from Lemma~\ref{lem: lower bound}.
\end{proof}
\def\proofname{Proof}

\section{Kei colorings of twist spun knots}
\label{sec: twist spun}

In this section, we consider the equality of Theorem~\ref{thm: lower bound} for 
the cases of twist spun knots and their stabilizations.

A {\it 1-tangle diagram} $T_K$ of a classical knot $K \subset S^3$
is a diagram of a tangle $K \setminus \mathring{B} \subset S^3 \setminus \mathring{B}$, where $B \subset S^3$ is a 3-ball such that $K \cap B$ is a trivial 1-tangle.
In \cite{AS05}, Asami and Satoh gave a broken surface diagram $\sD_m(K)$ of the $m$-twist-spinning $\mS_m(K)$ of $K$
such that there is an embedding $\iota_m \colon T_K \to \sD_{m}(K)$.
Then, for any kei $X$, the induced map 
$\iota^*_m \colon \Col_{X}(\sD_{m}(K)) \to \Col_{X}(T_K)$ is defined 
so that the coloring of each arc of $T_K$ is the same as that of the broken sheet containing the arc. Now, the next lemma directly follows from \cite[Lemma~5.1]{AS05}.
Here we choose an end point of $T_K$ and call the arc containing the end point {\it the terminal arc} of $T_K$. (Note that since we consider kei colorings and $T_K$ is unoriented, we can choose any end point of $T_K$ for defining the terminal arc.)
\begin{lem}[\text{\cite[Lemma~5.1]{AS05}}]
\label{lem: AS}
The map $\iota^*_m$ is injective. Moreover, $C \in \Col_{X}(T_K)$ belongs to 
$\iota^*_m(\Col_{X}(\sD_{m}(K)))$
if and only if $C * a^m_- = C$ holds, where $a_-$ is the color of the terminal arc of $T_K$.
\end{lem}
As a corollary of Lemma~\ref{lem: AS}, we have the following proposition.
\begin{prop}
\label{prop: AS}
We have the equality
\[
\iota^*_m(\Col_{X}(\sD_{m}(K)))
= 
\begin{cases}
\Col_{X}(T_K) & (\text{$m$ is even})\\
\{\text{trivial $X$-colorings of $T_K$}\} & (\text{$m$ is odd})
\end{cases}.
\]
\end{prop}
\begin{proof}
For any $a \in X$, the map $*a^2 \colon X \to X$ is the identity,
and hence we have the desired equality for the cases where $m$ is even. 
Next, for any odd $m$, we see that
\begin{eqnarray*}
\iota^*_m(\Col_{X}(\sD_{m}(K)))
&=&
\{C \in \Col_{X}(T_K) \mid C * a_-^m = C\}\\
&=&
\{C \in \Col_{X}(T_K) \mid C * a_- = C\}\\
&=& \iota^*_1(\Col_{X}(\sD_{1}(K))).
\end{eqnarray*}
Here we recall that $\sD_1(K)$ is a broken surface diagram of a $1$-twist spun $\mS_1(K)$,
which is trivial for any $K$; see \cite{Zee65}. In particular, $\sD_1(K)$ admits only trivial $X$-colorings, and this fact implies the desired equality for any odd $m$. 
\end{proof}
Now we prove the following theorem, which can be regarded as a generalization of Theorem~\ref{thm: twist spun}.
Here we recall that $P$ is either one of the two trivially embedded $\R P^2$'s in $S^4$ (with normal Euler number $\pm 2$) and $T$ a trivially embedded torus in $S^4$.
\begin{theorem}
\label{thm: twist spun general}
For a 1-tangle diagram $T_K$ of a classical knot $K$, suppose that there exists a finite kei $X$ with  
$b(K) = \log_{\# X}(\# \Col_{X}(T_K))$.
Then, for any $m \in \Z$, we have the following equalities:
\begin{itemize}
\item 
$b(\mS_{2m}(K) ) = 3 b(K) - 2$,
\item
$b(\mS_{2m}(K) \# P) = 3 b(K) -1$, and
\item
$b(\mS_{2m}(K) \# T) = 3 b(K)$.
\end{itemize}
\end{theorem}
\begin{proof}
Since a $(3b(K)-2)$-bridge trisection of $\mS_{2m}(K)$ is given in \cite{MZ17},
we have $b(\mS_{2m}(K) ) \leq 3 b(K) - 2$. Moreover, 
it is shown in Figure~15 and 17 of \cite{MZ17} that $b(P) =2$ and $b(T)=3$,
and hence it follows from
\cite[Subsection~2.2]{MZ17}
that
\[
b(\mS_{2m}(K) \# P) \leq b(\mS_{2m}(K)) + b(P) -1 \leq 3 b(K) -1
\]
and
\[
b(\mS_{2m}(K) \# T) \leq b(\mS_{2m}(K)) + b(T) -1 \leq 3 b(K).
\]


Next, since $P$ and $T$ have broken surface diagrams with a single broken sheet, 
it is obvious that 
\[
\# \Col_X(\mS_{2m}(K))= \# \Col_X(\mS_{2m}(K) \# P) = \# \Col_X(\mS_{2m}(K) \# T).
\]
Moreover, Proposition~\ref{prop: AS} gives
\[
\# \Col_X(\mS_{2m}(K)) = \# \Col_X(T_K).
\]
Combining these equalities with Theorem~\ref{thm: lower bound}, 
we have the opposite inequalities.
\end{proof}
Finally, we show that Theorem~\ref{thm: twist spun} follows from Theorem~\ref{thm: twist spun general}.
Here, we say that a kei $X$ is {\it faithful} if for any two distinct elements $a,b \in X$,
the bijections $*a$ and $*b$ are different maps. Then the following lemma is shown in \cite{Nos11, CSV16}.
\begin{lem}[\text{\cite[Lemma~5.6]{Nos11}, \cite[Lemma~4.4]{CSV16}}]
\label{lem: faithful}
For any faithful kei $X$, we have
$\Col_X(T_K) = \Col_X(K).$
\end{lem}

\def\proofname{Proof of Theorem~\ref{thm: twist spun}}
\begin{proof}
Note that $R_p$ is faithful for any odd $p > 1$. 
Therefore, by Lemma~\ref{lem: faithful}, the condition 
$b(K) = \log_{\# R_p} (\# \Col_{R_p}(T_K))$
is equivalent to $b(K)= \log_{\# R_p} (\# \Col_{R_p}(K))$.
Now, Theorem~\ref{thm: twist spun} immediately follows from
Theorem~\ref{thm: twist spun general}.
\end{proof}
\def\proofname{Proof}

\begin{center}
\scshape{Acknowledgments}
\end{center}
\vspace{4pt}
 
The first-named author has been supported by 
the Grant-in-Aid for JSPS Fellows (No.\,18J00808), Japan Society for the
Promotion of Science.
The second-named author has been supported in part by the Grant-in-Aid
for Scientific Research (C), (No.\,JP17K05242), Japan Society for the
Promotion of Science.

\bibliographystyle{hplain}
\bibliography{tex}

\begin{thebibliography}{10}

\bibitem{AS05}
Soichiro Asami and Shin Satoh.
\newblock An infinite family of non-invertible surfaces in 4-space.
\newblock {\em Bull. London Math. Soc.}, 37(2):285--296, 2005.

\bibitem{CS-book}
J.~Scott Carter and Masahico Saito.
\newblock {\em Knotted surfaces and their diagrams}, volume~55 of {\em
  Mathematical Surveys and Monographs}.
\newblock American Mathematical Society, Providence, RI, 1998.

\bibitem{CKS04}
Scott Carter, Seiichi Kamada, and Masahico Saito.
\newblock {\em Surfaces in 4-space}, volume 142 of {\em Encyclopaedia of
  Mathematical Sciences}.
\newblock Springer-Verlag, Berlin, 2004.
\newblock Low-Dimensional Topology, III.

\bibitem{CSV16}
W.~Edwin Clark, Masahico Saito, and Leandro Vendramin.
\newblock Quandle coloring and cocycle invariants of composite knots and
  abelian extensions.
\newblock {\em J. Knot Theory Ramifications}, 25(5):1650024, 34, 2016.

\bibitem{Fox}
R.~H. Fox.
\newblock A quick trip through knot theory.
\newblock In {\em Topology of 3-manifolds and related topics ({P}roc. {T}he
  {U}niv. of {G}eorgia {I}nstitute, 1961)}, pages 120--167. Prentice-Hall,
  Englewood Cliffs, N.J., 1962.

\bibitem{Fox2}
R.~H. Fox.
\newblock Metacyclic invariants of knots and links.
\newblock {\em Canadian J. Math.}, 22:193--201, 1970.

\bibitem{Kam-book}
Seiichi Kamada.
\newblock {\em Braid and knot theory in dimension four}, volume~95 of {\em
  Mathematical Surveys and Monographs}.
\newblock American Mathematical Society, Providence, RI, 2002.

\bibitem{MZ17}
Jeffrey Meier and Alexander Zupan.
\newblock Bridge trisections of knotted surfaces in {$S^4$}.
\newblock {\em Trans. Amer. Math. Soc.}, 369(10):7343--7386, 2017.

\bibitem{Nos11}
Takefumi Nosaka.
\newblock On homotopy groups of quandle spaces and the quandle homotopy
  invariant of links.
\newblock {\em Topology Appl.}, 158(8):996--1011, 2011.

\bibitem{Ros98}
Witold Rosicki.
\newblock Some simple invariants of the position of a surface in {$\bold R^4$}.
\newblock {\em Bull. Polish Acad. Sci. Math.}, 46(4):335--344, 1998.

\bibitem{Tak43}
Mituhisa Takasaki.
\newblock Abstraction of symmetric transformations.
\newblock {\em T\^{o}hoku Math. J.}, 49:145--207, 1943.

\bibitem{Zee65}
E.~C. Zeeman.
\newblock Twisting spun knots.
\newblock {\em Trans. Amer. Math. Soc.}, 115:471--495, 1965.

\end{thebibliography}

\end{document}